\documentclass[a4paper,10pt]{article}

\usepackage[square,numbers]{natbib}
\usepackage{amsfonts,amsmath,amssymb,amsthm}
\usepackage{hyperref}
\usepackage{graphicx}
\usepackage[font=small,labelfont=bf]{caption}
\usepackage{helvet}

\usepackage{tikz}
\usetikzlibrary{arrows,decorations,shapes,positioning}

\theoremstyle{definition}
\newtheorem{theorem}{Theorem}[section]

\newtheorem{prop}[theorem]{Proposition}
\newtheorem{defn}[theorem]{Definition}

\newtheorem{rem}[theorem]{Remark}
\newtheorem{ex}[theorem]{Example}

\newcommand{\BMD}{M}
\newcommand{\BMDmod}{\widetilde{M}}

\title{Variations on Baur--Marsh's determinant}
\author{Philipp Lampe}

\begin{document}

\maketitle

\abstract{Baur and Marsh computed the determinant of a matrix assembled from the cluster variables in a cluster algebra of type A. In this article we wish to describe two variations. On the one hand, we compute determinants of matrices assembled from the squares of the cluster variables in Baur--Marsh's matrix. One such determinant admits an interpretation as a Cayley--Menger determinant. On the other hand, we wish to present a formula for the determinant of a matrix of cluster variables in a cluster algebra of type D. This cluster algebra is associated with a marked oriented surface. As in Baur--Marsh's setup the matrix is indexed by the marked points of the surface and an entry is given by the cluster variable corresponding to an arc between two marked points. Our formula asserts that the determinant may again be written as a product of cluster variables.

\section{Introduction}

Cluster theory is a newish and active mathematical domain with connections to various branches of mathematics such as root systems, geometry, and algebraic combinatorics. It was initiated by Sergey Fomin and Andrei Zelevinsky in a seminal treatise \cite{FZ}, and we will give a short introduction to cluster theory in Section \ref{Sec:CA}.\let\thefootnote\relax\footnote{This work was supported by EPSRC grant EP/N005457/1.}

An illustrative example is the cluster algebra $\mathcal{A}(A_n)$ (for some positive integer n) which admits several equivalent descriptions. First, it may be seen as a finite dynamical system of tuples of Laurent polynomials attached to a finite root system of type $A_n$. Second, Fomin--Shapiro--Thurston \cite{FST} give a geometric description of the cluster algebra $\mathcal{A}(A_n)$ in terms of an $(n+3)$-gon with vertices $0,1,\ldots,n+2$. Here the cluster variables $x_{ij}=x_{ji}\in\mathcal{A}(A_n)$ are parametrized by two different indices $i,j\in\{0,1,\ldots,n+2\}$ and the relations among the cluster variables are analogues of Ptolemy's relation for quadrilaterals. The relations state that $x_{ik}\cdot x_{jl}=x_{ij}\cdot x_{kl}+x_{jk}\cdot x_{il}$ whenever $i<j<k<l$. Third, from the perspective of algebraic combinatorics we consider a $2\times (n+3)$ matrix $X$ assembled from $2(n+3)$ formal variables. Then the cluster variables $x_{ij}$ with $i<j$ can be realized as the $2\times 2$ minors of $X$ supported on columns $i$ and $j$. In this context the above relations among the cluster variables are known as Pl\"ucker relations.

Karin Baur and Robert Marsh \cite{BM} proved the following theorem which shall be the starting point of this paper. It is presented in more detail as Theorem \ref{thm:BaurMarsh} in Section \ref{Sec:Det}.

\begin{theorem}[Baur--Marsh]
Let $BM$ be the symmetric $(n+3)\times(n+3)$ matrix with entries $BM_{ij}=x_{ij}\in\mathcal{A}(A_n)$ (where we interpret the diagonal entries $x_{ii}$ as zero). The determinant satisfies the relation:
\begin{align*}
\operatorname{det}(BM)=-(-2)^{n+1}\cdot x_{0,1}\cdot x_{1,2}\cdot x_{2,3}\cdot\ldots\cdot x_{n+1,n+2}\cdot x_{n+2,0}.
\end{align*}
\end{theorem}

The purpose of this note is to discuss some related determinants. In the first instance we keep the ground ring $\mathcal{A}(A_n)$, and vary the construction of the matrix in analogy with a determinant named after Arthur Cayley and Karl Menger:  

\begin{theorem}[Cayley--Menger cluster matrix of type A]
Let $PM$ be the symmetric $(n+3)\times(n+3)$ matrix with entries $PM_{ij}=x^2_{ij}\in\mathcal{A}(A_n)$. Moreover, let $e=(1,1,\ldots,1)^T\in\mathbb{Z}^{n+3}$ be the column vector of all ones. 
We refer to the symmetric $(n+4)\times(n+4)$ matrix
\begin{align*}
CM=\left(
\begin{matrix}
PM & e\\
e^T &0
\end{matrix}
\right).
\end{align*}
with entries in $\mathcal{A}(A_n)$ as the Cayley--Menger cluster matrix of type $A_n$. If $n\geq 2$, then the matrix $CM$ and its principal submatrix $PM$ satisfy the relation $\operatorname{det}(PM)=\operatorname{det}(CM)=0$. 
\end{theorem}

The theorem is presented in more detail as Theorem \ref{Thm:CM} in Section \ref{Sec:Det}. For a proof we will fix a vertex $r$ and substitute $\widetilde{x}_{ij}=x_{ij}/(x_{ri}\cdot x_{rj})$ for all indices $i,j$ different from $r$. In this way we transform the exchange relations into linear equations of the form $\widetilde{x}_{ij}+\widetilde{x}_{jk}=\widetilde{x}_{ik}$ for all indices $i,j,k$ such that $r,i,j,k$ are pairwise different and lie on the boundary of the polygon in this order. These relations will allow us to write the matrix as a linear combination of rank $1$ matrices.

Now let us describe a result for the cluster algebra $\mathcal{A}(D_n)$ attached to a finite root system of type $D_n$. It can be modelled geometrically by an $n$-gon $\Sigma$ with vertices $1,\ldots,n$ together with a puncture $0$ in its interior. Here the cluster variables are indexed by tagged arcs, and for two different marked points $i,j\in\{0,1,\ldots,n\}$ there are exactly two tagged arcs between $i$ and $j$. Let $i,j\geq 1$ be marked points on the boundary. Suppose that the two (necessarily plain tagged) arcs $\alpha=\operatorname{im}(a)$ and $\beta=\operatorname{im}(b)$ between $i$ and $j$ are realized by curves $a,b\colon[0,1]\to\Sigma$ with $a(0)=b(1)=i$ and $a(1)=b(0)=j$ such that the concatenation $a \ast b$ is a positively oriented loop around $i$. Then we denote the cluster variable $x_{\alpha}\in\mathcal{A}(D_n)$ by $x_{ij}$ and the cluster variable $x_{\beta}\in\mathcal{A}(D_n)$ by $x_{ji}$. Moreover, we denote by $x_{0i}\in\mathcal{A}(D_n)$ the cluster variable associated with the plain tagged arc between $0$ and $i$ and by $x_{i0}\in\mathcal{A}(D_n)$ the cluster variable associated with the notched tagged arc between $0$ and $i$.

\begin{theorem}[Generalized Baur--Marsh matrix of type D] Let $\BMD$ be the symmetric $(n+1)\times(n+1)$ matrix with entries $\BMD_{ij}=x_{ij}\in\mathcal{A}(D_n)$ (where we interpret the diagonal entries $x_{ii}$ as zero). Then the determinant obeys the relation:
\begin{align*}
\operatorname{det}(\BMD)= (-1)^{n}\cdot \prod_{i=1}^{n} x_{0i}\cdot x_{i0}.
\end{align*}
\end{theorem}
The theorem is presented in more detail as Theorem \ref{Thm:TypeD} in Section \ref{Sec:Det}. We give an elementary proof by performing row operations. Note that we may interpret the factor $x_{i0}\cdot x_{0i}$ on the right hand side of the theorem as the lambda length of the loop around the puncture $0$ with endpoint $i$.

\section{Surface cluster algebras}
\label{Sec:CA}

\subsection{Introduction to Fomin--Zelevinsky cluster algebras}
Cluster algebra theory started in 2002 with a seminal article by Fomin and Zelevinsky \cite{FZ}. By construction the \textit{cluster algebra} is a $\mathbb{Z}$-algebra defined by generators and relations. Let us recall the main steps in the construction of the generators and relations. First of all, we fix a natural number $n$ which is called the \textit{rank} of the cluster algebra. The distinguished generators of a cluster algebra are called \textit{cluster variables}, and certain tuples $\mathbf{x}=(x_1,x_2,\ldots,x_n)$ of cluster variables of cardinality $n$ are called \textit{clusters}. Moreover, every cluster $\mathbf{x}$ is endowed with a skew-symmetric (or more generally a skew-symmetrizable) integer $n\times n$ matrix $B=(b_{ij})_{1\leq i,j\leq n}$ which is called \textit{exchange matrix}. In this case, the pair $(\mathbf{x},B)$ is called a \textit{seed}. An essential notion is the \textit{mutation of seeds}. Here, a mutation of a seed $(\mathbf{x},B)$ at an index $k\in\{1,2,\ldots,n\}$ is a seed $\mu_k(\mathbf{x},B)=(\mathbf{x}',B')$ such that the new cluster $\mathbf{x}'=\left(x_1,\ldots,x_{k-1},x_k',x_{k+1},\ldots,x_n\right)$ is obtained from the old cluster $\mathbf{x}$ by replacing the cluster variable $x_k$ with another cluster variable $x_k'$. Both cluster variables are associated with each other by an exchange relation of the form 
\begin{align*}
x_k\cdot x_k'=\prod_{\genfrac{}{}{0pt}{}{1\leq i\leq n}{b_{ik}>0}}x_i^{b_{ik}}+\prod_{\genfrac{}{}{0pt}{}{1\leq j\leq n}{b_{jk}<0}}x_j^{-b_{jk}}
\end{align*}
where the exponents are the entries in the corresponding column of the matrix $B$. Moreover, there is a combinatorial rule relating the exchange matrices $B$ and $B'$. In this presentation we forego the definition of mutation of exchange matrices. Nevertheless let us remark that the construction is made in such a way that every mutation $\mu_k$ is an involution, i.\,e. the equation $(\mu_k\circ\mu_k)(\mathbf{x},B)=(\mathbf{x},B)$ holds for all seeds $(\mathbf{x},B)$ and all indices $k$. In the above situation we also write $B'=\mu_k(B)$. Given an initial seed $(\mathbf{x},B)$, the cluster algebra $\mathcal{A}(\mathbf{x},B)\subseteq\mathbb{Q}(x_1,x_2,\ldots,x_n)$ is defined to be the subalgebra generated by all cluster variables that belong to seeds generated from $(\mathbf{x},B)$ by sequences of mutations. In this situation the field $\mathbb{Q}(x_1,x_2,\ldots,x_n)$ of rational functions in the initial cluster variables is called the \textit{ambient field} of the cluster algebra $\mathcal{A}(\mathbf{x},B)$. Sometimes we freeze a subset of indices $F\subseteq \{1,2,\ldots,n\}$ and only allow mutations at non-frozen indices in $I=\{1,2,\ldots,n\}\backslash F$. In this case it is enough to keep track of the submatrix of the exchange matrix on rows $\{1,2,\ldots,n\}$ and columns $I$.

A main theorem about cluster algebras is Fomin--Zelevinsky's Laurent phenomenon \cite{FZ}, which asserts that every cluster variable in a cluster algebra $\mathcal{A}(\mathbf{x},B)$ can be written as a Laurent polynomial in the cluster variables $x_1,x_2,\ldots,x_n$ of the initial cluster $\mathbf{x}$. Another main theorem is Fomin--Zelevinsky's classification \cite{FZ2} of cluster algebras with finitely many cluster variables by finite type root systems. If a cluster algebra $\mathcal{A}(\mathbf{x},B)$ admits only finitely many cluster variables, then the non-initial cluster variables are in bijection with the positive roots in a finite root system. More precisely, by the Laurent phenomenon every non-initial cluster variable may be written as $P(x_1,x_2,\ldots,x_n)/(x_1^{a_1}\cdot x_2^{a_2}\cdot\ldots\cdot x_n^{a_n})$ with a polynomial $P\in\mathbb{Z}[x_1,x_2,\ldots,x_n]$ and natural numbers $a_1,a_2,\ldots,a_n$. The bijection assigns to this cluster variable the positive root $a_1\alpha_1+a_2\alpha_2+\ldots+a_n\alpha_n$ where $\alpha_1,\alpha_2,\ldots,\alpha_n$ are the simple roots in the corresponding root system. Recall that there are seven different types of finite root systems called $A,B,C,D,E,F,G$. In this language cluster algebras of type $A,D$ or $E$ arise from skew-symmetric exchange matrices whereas cluster algebras of type $B,C,F$ or $G$ arise from skew-symmetrizable but not skew-symmetric exchange matrices. A third main theorem is Felikson--Shapiro--Tumarkin's classification \cite{FST} of cluster algebras with finitely many exchange matrices. 

Cluster algebra structures appear in representation theory, topology, geometry and number theory. For example, in the context of representation theory of Lie algebras Lusztig's canonical bases often carry cluster algebra structures. Canonical bases were one of the motivating examples to introduce the formalism of cluster algebras. Another important instance of cluster algebras arises in the context of geometry and topology. Building on work of Gekhtman--Shapiro--Vainshtein \cite{GSV}, Fock--Goncharov \cite{FG1,FG2} and Penner \cite{P2}, Fomin--Shapiro--Thurston \cite{FST,FT} constructed cluster algebra structures from marked oriented surfaces. Surface cluster algebras are interesting for many reasons, for example due to their prominent role in Felikson--Shapiro--Tumarkin's classification of cluster algebras with finitely many exchange matrices. Here, the mutation class of an exchange matrix attached to a surface is always mutation-finite and so is every $2\times 2$ exchange matrix. The classification implies that aside from these there are only eight mutation-finite skew-symmetric, connected exchange matrices up to mutation-equivalence.

\subsection{Marked oriented surfaces and their arcs and triangulations}

In the next sections we review Fomin--Shapiro--Thurston's construction of cluster algebras from surfaces \cite{FST,FT}. Let $\Sigma$ be a connected oriented surface, possibly with boundary. Up to homeomorphism we may assume that $\Sigma\cong\Sigma_{g,b}$ is a connected closed surface $\Sigma_{g}$ of genus $g$ with $b$ disks removed. A particular case of interest is the surface $\Sigma=\Sigma_{0,1}$, where $\Sigma\cong D$ is homeomorphic to a disk. Let $M\subseteq\overline{\Sigma}$ be a finite nonempty set of points in the closure of $\Sigma$ such that every boundary component of $\Sigma$ contains at least one point in $M$. The elements in $M$ are called \textit{marked points}, marked points in the interior of $\Sigma$ are called \textit{punctures}, and the pair $(\Sigma,M)$ is called a \textit{marked oriented surface}. When $\Sigma=\Sigma_{0,1}\cong D$ is a disk and the number of marked points on the boundary $\lvert M\cap\partial \Sigma\rvert =n$, we call the pair $(\Sigma,M)$ an $n$-\textit{gon}. Moreover, we refer to the special cases $n=1,2,3,4$ as \textit{monogon}, \textit{digon}, \textit{triangle} and \textit{quadrilateral}. To avoid technical difficulties we impose the following conditions: 
\begin{itemize}
\item[(a)] If $\Sigma=\Sigma_{0,0}$ is a sphere, then $M$ contains at least $4$ punctures.
\item[(b)] If $(\Sigma,M)$ is a monogon, then $M$ contains at least $2$ punctures.
\item[(c)] If $(\Sigma,M)$ is a digon or a triangle, then $M$ contains at least $1$ puncture.
\end{itemize}
The following definitions are crucial:

\begin{defn}[Arcs]
The image $\alpha=\operatorname{im}(a)$ of a curve $a\colon [0,1]\to \Sigma$ is called an \textit{arc} in $(\Sigma,M)$ if the following conditions hold: 
\begin{itemize}
\item[(a)] The map $a$ restricts to an injective map $a\vert_{(0,1)}\colon (0,1)\to \Sigma\backslash M$.
\item[(b)] The endpoints $a(0)$ and $a(1)$ lie in $M$.
\item[(c)] The arc $\alpha$ does not cut out an unpunctured monogon or an unpunctured digon. 
\end{itemize}
\end{defn}

\begin{defn}[Properties of arcs] An arc $\alpha=\operatorname{im}(a)$ is called a \textit{loop} if the endpoints $a(0)=a(1)$ coincide. Two arcs $\alpha$ and $\beta$ are called \textit{isotopic} if $\alpha=\operatorname{im}(a)$ and $\beta=\operatorname{im}(b)$ for some curves $a,b$ that are homotopic to each other in $\Sigma\backslash M$ relative to the endpoints $a(0)=b(0)$ and $a(1)=b(1)$. An arc $\alpha=\operatorname{im}(a)$ is called a \textit{boundary arc} if it is isotopic to an arc in $\partial \Sigma$. Otherwise it is called an \textit{inner arc}. We say that two arcs $\alpha,\beta$ are \textit{non-crossing} if $\alpha$ is isotopic to some arc $\alpha'$ and $\beta$ is isotopic to some arc $\beta'$ such that $\alpha'$ and $\beta'$ do not intersect except possibly at endpoints. They are called \textit{crossing} otherwise. 
\end{defn}

Often we do not distinguish between isotopic arcs. 

\begin{defn}[Triangulations] A \textit{triangulation} $\mathcal{T}$ of the marked oriented surface $(\Sigma,M)$ is a maximal collection of pairwise non-isotopic arcs with disjoint interiors. The connected components of the complement $\Sigma\backslash\bigcup_{\alpha\in \mathcal{T}}\{\alpha\}$ are called \textit{triangles} of $\mathcal{T}$. A triangle bounded by only two distinct sides as in the left hand side of Figure \ref{Figure:SelfFoldedTriangle} is called \textit{self-folded}. In this situation we call the arc $\alpha$ with endpoints $0,1$ the \textit{radius} of the self-folded triangle.
\end{defn}

\begin{figure}
\begin{center}
\scalebox{.9}{
\begin{tikzpicture}

    \newcommand{\verticaldistance}{1.7}
    \newcommand{\horizontaldistance}{2.5}
    \newcommand{\labeldistance}{0.05}
    
    \node (A) at (-\horizontaldistance,0) {$\bullet$};
    \node (B) at (-\horizontaldistance,-\verticaldistance) {$\bullet$};
    
    \node (C) [above=\labeldistance of A] {$0$};
    \node (D) [below=\labeldistance of B] {$1$};
    \node (P) at (-\horizontaldistance+6*\labeldistance,-0.5*\verticaldistance) {$\alpha$};
    \node (R) at (-\horizontaldistance,-2*\verticaldistance-10*\labeldistance) {$\eta$};
    
    \path[thick] (-\horizontaldistance,0) edge (-\horizontaldistance,-\verticaldistance);
    \draw[thick] (-\horizontaldistance,-\verticaldistance) ellipse (2 and \verticaldistance);
    
    \node (E) at (\horizontaldistance,0) {$\bullet$};
    \node (F) at (\horizontaldistance,-\verticaldistance) {$\bullet$};
    
    \node (G) [above=\labeldistance of E] {$0$};
    \node (H) [below=\labeldistance of F] {$1$};
    \node (I) [above=\labeldistance of F] {$\bowtie$};
    \node (Q) at (\horizontaldistance+6*\labeldistance,-0.5*\verticaldistance) {$\overline{\alpha}$};
    \node (S) at (\horizontaldistance,-2*\verticaldistance-10*\labeldistance) {$\eta$};
    
    \path[thick] (\horizontaldistance,0) edge (\horizontaldistance,-\verticaldistance);
     \draw[thick] (\horizontaldistance,-\verticaldistance) ellipse (2 and \verticaldistance);
     
    \node (T) [inner sep=0pt,outer sep=0pt] at (3*\horizontaldistance,0) {$\bullet$};
    \node (U) at (3*\horizontaldistance,-\verticaldistance) {$\bullet$};
    \node (V) at (3*\horizontaldistance,-2*\verticaldistance) {$\bullet$};
    
    \node (W) [above=\labeldistance of T] {$1$};
    \node (X) [right=\labeldistance of U] {$0$};
    \node (Y) [below=\labeldistance of V] {$2$};
    
    \path[thick] (3*\horizontaldistance,0) edge (3*\horizontaldistance,-2*\verticaldistance);
    \draw[thick] (3*\horizontaldistance,-\verticaldistance) ellipse (2 and \verticaldistance);
    \draw[thick] (T) to [out=225,in=315,looseness=75] (T);
    
    \node (a) at (3*\horizontaldistance+2+6*\labeldistance,-\verticaldistance) {$\alpha$};
    \node (b) at (3*\horizontaldistance-2-6*\labeldistance,-\verticaldistance) {$\beta$};
    \node (c) at (3*\horizontaldistance+6*\labeldistance,-0.6*\verticaldistance) {$\gamma$};
    \node (d) at (3*\horizontaldistance+6*\labeldistance,-1.7*\verticaldistance) {$\delta$};
    \node (e) at (3*\horizontaldistance+20*\labeldistance,-1.4*\verticaldistance) {$\eta$};

\end{tikzpicture}}
\captionsetup{width=0.8\textwidth}
\caption{A self-folded triangle with a plain tagged arc (left), a self-folded triangle with a notched tagged arc (middle), and arcs in a punctured digon (right)}
\label{Figure:SelfFoldedTriangle}
\end{center}
\end{figure}
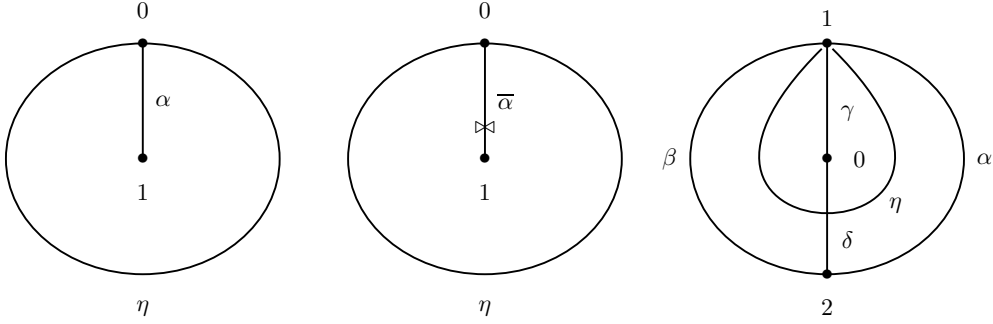

Every triangulation $\mathcal{T}$ may be written as the union $\mathcal{T}=\mathcal{T}^i\cup\mathcal{T}^b$ of the set of inner arcs of $\mathcal{T}$ and the set of boundary arcs of $\mathcal{T}$. Note that every triangulation contains all boundary arcs.

\subsection{Exchange matrices associated with triangulated surfaces}

Let $\mathcal{T}$ be a triangulation of a marked oriented surface $(\Sigma,M)$. If $\alpha\in\mathcal{T}$ is neither the radius of a self-folded triangle nor a boundary arc, then up to isotopy there are exactly two arcs that can be appended to $\mathcal{T}\backslash\{\alpha\}$ to form a triangulation, namely $\alpha$ itself and another arc $\alpha'$ not isotopic to $\alpha$. 

\begin{defn}[Flips of triangulations]
In the above situation we say that the arc $\alpha$ is \textit{flippable} and we define the \textit{flip} of the triangulation $\mathcal{T}$ at the flippable arc $\alpha$ to be $F_{\alpha}(\mathcal{T})=(\mathcal{T}\backslash\{\alpha\})\cup\{\alpha'\}$. 
\end{defn}

\begin{defn}[Exchange matrices of triangulations]
For every not self-folded triangle $\Delta$ in the complement $\Sigma\backslash\bigcup_{\alpha\in \mathcal{T}}\{\alpha\}$ we construct a skew-symmetric matrix $B^{\Delta}=(b^{\Delta}_{\alpha,\beta})$ indexed by $\mathcal{T}\times\mathcal{T}$. The entries of $B^{\Delta}$ are zero except for $b_{\gamma,\delta}^{\Delta}=b_{\delta,\epsilon}^{\Delta}=b_{\epsilon,\gamma}^{\Delta}=1$ and $b_{\delta,\gamma}^{\Delta}=b_{\epsilon,\delta}^{\Delta}=b_{\gamma,\epsilon}^{\Delta}=-1$, where $\gamma$, $\delta$ and $\epsilon$ are the three elements in $\mathcal{T}$ obtained by parsing the positive orientation of the boundary $\partial \Delta$ in this order. We define $B^{\mathcal{T}}$ to be the submatrix indexed by $\mathcal{T}\times\mathcal{T}^i$ of the matrix $\sum B^{\Delta}$ where the sum runs over all not self-folded triangles $\Delta\in\Sigma\backslash\bigcup_{\alpha\in \mathcal{T}}\{\alpha\}$. 
\end{defn}

Note that the submatrix of $B^{\mathcal{T}}$ indexed by $\mathcal{T}^i\times\mathcal{T}^i$ is skew-symmetric. Hence $\mathcal{B}^{\mathcal{T}}$ is an exchange matrix in the sense of Fomin--Zelevinsky. A crucial observation asserts that if $\mathcal{T}'=F_{\alpha}(\mathcal{T})$ is the flip of $\mathcal{T}$ at a flippable arc $\alpha$, then the matrix $B^{\mathcal{T}'}$ agrees with the mutated exchange matrix $\mu_{\alpha}(B^{\mathcal{T}})$. To extend the notion of a flip to radii, Fomin--Shapiro--Thurston introduced a tagging of arcs:

\begin{defn}[Tagged arcs] A \textit{tagged arc} $(\alpha,t)$ is an ordinary arc $\alpha=\operatorname{im}(a)$ together with a tagging $t$ of the endpoints $a(0)$ and $a(1)$ as \textit{plain} or \textit{notched} such that the following conditions hold: 
\begin{itemize}
\item[(a)] A marked point on the boundary $\partial \Sigma$ receives only plain labels. 
\item[(b)] If $\alpha$ is a loop, then its endpoint receives the same label on both sides: $t(a(0))=t(a(1))$.
\item[(c)] The arc does not cut out a monogon with exactly one puncture. 
\end{itemize}
\end{defn}
To every ordinary arc $\alpha$ we associate a tagged arc $\tau(\alpha)$ in the following way: Since $(\Sigma,M)$ is not a sphere with $1$, $2$ or $3$ punctures, an arc $\alpha$ cuts out at most one monogon with exactly one puncture. Firstly, if $\alpha$ does not cut out a monogon with exactly one puncture, then $\tau(\alpha)$ is equal to $\alpha$ with both endpoints tagged plain. Secondly, if a loop $\alpha$ around a marked point $0$ does cut out a monogon with exactly one puncture $1$, then $\tau(\alpha)$ is equal to the radius of the corresponding self-folded triangle where $0$ is tagged plain and $1$ is tagged notched. For this reason we often we omit the tagging $t$ and write $\alpha$ instead of $(\alpha,t)$ for brevity. Moreover, we identify tagged arcs whenever the underlying ordinary arcs are isotopic and the common endpoints receive the same labels.

\begin{defn}[Compatibility of tagged arcs and tagged triangulations] \begin{itemize}
\item[(a)] Two tagged arcs $\alpha$ and $\beta$ are called \textit{compatible} if the following conditions hold: 
\begin{itemize}
\item[(i)] The underlying ordinary arcs of $\alpha$ and $\beta$ are non-crossing.
\item[(ii)] If the underlying ordinary arcs of $\alpha$ and $\beta$ are not isotopic and have an endpoint $P$ in common, then $P$ receives the same label from $\alpha$ and $\beta$. 
\item[(iii)] If the underlying ordinary arcs of $\alpha$ and $\beta$ are isotopic, then at least one endpoint receives the same label from both $\alpha$ and $\beta$.
\end{itemize}
\item[(b)] A \textit{tagged triangulation} of the marked oriented surface $(\Sigma,M)$ is a maximal collection of pairwise non-isotopic and compatible tagged arcs. 
\end{itemize}
\end{defn}

It can be shown that all tagged triangulations of $(\Sigma,M)$ have the same cardinality and contain all the (necessarily plain-tagged) boundary arcs. Moreover, if a tagged inner $\alpha\in\mathcal{T}$ is contained in a tagged triangulation, then there is a unique tagged arc $\alpha'$ different from $\alpha$ which can be appended to $\mathcal{T}\backslash\{\alpha\}$ to form a tagged triangulation. We define the \textit{flip} of $\mathcal{T}$ at $\alpha$ as $F_{\alpha}(\mathcal{T})=(\mathcal{T}\backslash\{\alpha\})\cup\{\alpha'\}$. 

\begin{ex}[Once punctured triangle] Let us consider a triangle with one puncture $0$ and three marked points $1,2,3$ on the boundary. We can embed the surface in $\mathbb{R}^2$ and orient it counterclockwise. The left hand side of Figure \ref{Figure:Triangulation} displays a triangulation of $(\Sigma,M)$ with inner arcs $\alpha$, $\beta$, $\gamma$ and boundary arcs $\delta$, $\epsilon$, $\digamma$. The triangulation admits three triangles $\Delta_1=(023)$, $\Delta_2=(031)$ and $\Delta_3=(012)$ none of which is self-folded. Parsing the boundary $\partial \Delta_1$ in counterclockwise order yields the cycle $\beta\to\delta\to\gamma\to\beta$ so that $b_{\beta,\delta}^{\Delta_1}=b_{\delta,\gamma}^{\Delta_1}=b_{\gamma,\beta}^{\Delta_1}=1$ and $b_{\delta,\beta}^{\Delta_1}=b_{\gamma,\delta}^{\Delta_1}=b_{\beta,\gamma}^{\Delta_1}=-1$. We obtain
\begin{align*}
B^{\mathcal{T}}=
\left(\begin{matrix}0&0&0\\0&0&-1\\0&1&0\\0&-1&1\\0&0&0\\0&0&0\end{matrix}\right)+
\left(\begin{matrix}0&0&1\\0&0&0\\-1&0&0\\0&0&0\\1&0&-1\\0&0&0\end{matrix}\right)+
\left(\begin{matrix}0&-1&0\\1&0&0\\0&0&0\\0&0&0\\0&0&0\\-1&1&0\end{matrix}\right)=
\left(\begin{matrix}0&-1&1\\1&0&-1\\-1&1&0\\0&-1&1\\1&0&-1\\-1&1&0\end{matrix}\right)
\end{align*}  
The flip of $\mathcal{T}$ at arc $\alpha$ produces a plain-tagged arc $\alpha'$ with endpoints $2$ and $3$ which is shown on the right hand side of Figure \ref{Figure:Triangulation}. 
\end{ex}

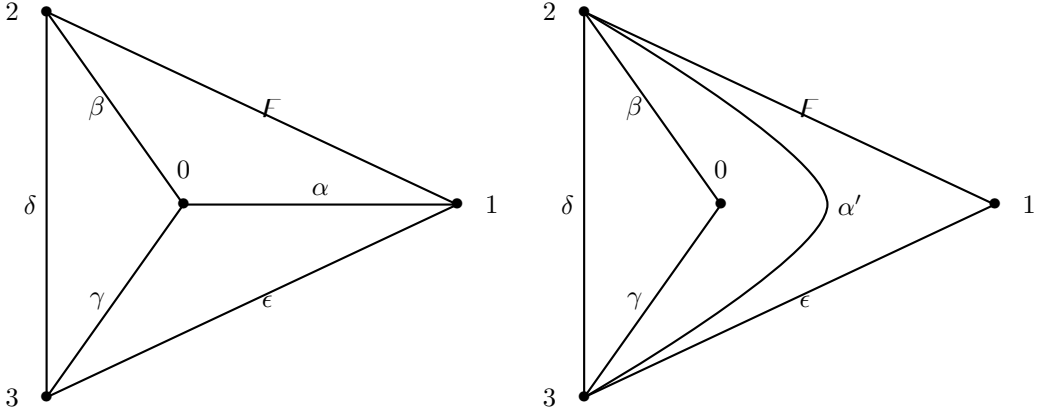
\begin{figure}
\begin{center}
\begin{tikzpicture}

    \newcommand{\length}{3.6}
    \newcommand{\labeldistance}{0.3mm}

    \node (0) at (0,0) {$\bullet$};
    \node (1) at (\length,0) {$\bullet$};
    \node (2) at (-0.5*\length,0.71*\length) {$\bullet$}; 
    \node (3) at (-0.5*\length,-0.71*\length) {$\bullet$}; 
    
    \node (L0) [above=\labeldistance of 0] {$0$};
    \node (L1) [right=\labeldistance of 1] {$1$};
    \node (L2) [left=\labeldistance of 2] {$2$};
    \node (L3) [left=\labeldistance of 3] {$3$};
    
    \path[-,thick] (0,0) edge node[above] {$\alpha$} (\length,0); 
    \path[-,thick] (0,0) edge node[left] {$\beta$} (-0.5*\length,0.71*\length); 
    \path[-,thick] (0,0) edge node[left] {$\gamma$} (-0.5*\length,-0.71*\length); 
    
    \path[-,thick] (-0.5*\length,0.71*\length) edge node[left] {$\delta$} (-0.5*\length,-0.71*\length); 
    \path[-,thick]  (\length,0) edge node[below,right] {$\epsilon$}(-0.5*\length,-0.71*\length);
    \path[-,thick] (-0.5*\length,0.71*\length) edge node[above,right] {$\digamma$} (\length,0);  

\end{tikzpicture}\quad
\begin{tikzpicture}

    \newcommand{\length}{3.6}
    \newcommand{\labeldistance}{0.3mm}
    \newcommand{\bending}{1.4}

    \node (0) at (0,0) {$\bullet$};
    \node (1) at (\length,0) {$\bullet$};
    \node (2) at (-0.5*\length,0.71*\length) {$\bullet$}; 
    \node (3) at (-0.5*\length,-0.71*\length) {$\bullet$}; 
    
    \node (L0) [above=\labeldistance of 0] {$0$};
    \node (L1) [right=\labeldistance of 1] {$1$};
    \node (L2) [left=\labeldistance of 2] {$2$};
    \node (L3) [left=\labeldistance of 3] {$3$};
    
    \node (A) at (\bending+0.3,0) {$\alpha'$};
    
    \draw[thick] plot [smooth, tension=0.5] coordinates { (-0.5*\length,0.71*\length) (\bending,0) (-0.5*\length,-0.71*\length)};
    \path[-,thick] (0,0) edge node[left] {$\beta$} (-0.5*\length,0.71*\length); 
    \path[-,thick] (0,0) edge node[left] {$\gamma$} (-0.5*\length,-0.71*\length); 
    
    \path[-,thick] (-0.5*\length,0.71*\length) edge node[left] {$\delta$} (-0.5*\length,-0.71*\length); 
    \path[-,thick]  (\length,0) edge node[below,right] {$\epsilon$}(-0.5*\length,-0.71*\length);
    \path[-,thick] (-0.5*\length,0.71*\length) edge node[above,right] {$\digamma$} (\length,0);  

\end{tikzpicture}
\end{center}
\caption{Two triangulations of a marked oriented surface}
\label{Figure:Triangulation}
\end{figure}

\subsection{Seeds of surface cluster algebras}

With the marked oriented surface $(\Sigma,M)$ Fomin--Shapiro--Thurston associate a cluster algebra $\mathcal{A}(\Sigma,M)$. Here, the cluster variables are in bijection with the tagged arcs of $(\Sigma,M)$. In particular, with every tagged arc $\alpha$ we associate a cluster variable $x_{\alpha}\in\mathcal{A}(\Sigma,M)$. We say that the tuple $\mathbf{x}=(x_{\alpha_1},x_{\alpha_2},\ldots, x_{\alpha_n})$ of cluster variables forms a cluster if the corresponding tagged arcs $\alpha_1,\alpha_2,\ldots, \alpha_n$ form a tagged triangulation $\mathcal{T}$ of $(\Sigma,M)$. We complete $\mathbf{x}$ by the exchange matrix $B^{\mathcal{T}}$ to obtain a seed $(\mathbf{x},B^{\mathcal{T}})$. A fundamental observation asserts that mutations in $\mathcal{A}(\Sigma,M)$ are induced by flips $F_{\alpha}(\mathcal{T})=(\mathcal{T}\backslash\{\alpha\})\cup\{\alpha'\}$ of tagged triangulations $\mathcal{T}$ at tagged inner arcs $\alpha$. In particular, the exchange relation attached to the above flip is the polynomial relation
\begin{align}
\label{Exchange}
x_{\alpha}\cdot x_{\alpha'}=\prod_{\genfrac{}{}{0pt}{}{\beta\in\mathcal{T}}{b^{\mathcal{T}}_{\alpha,\beta}>0}} x_{\beta}^{b^{\mathcal{T}}_{\alpha,\beta}}+\prod_{\genfrac{}{}{0pt}{}{\gamma\in\mathcal{T}}{b^{\mathcal{T}}_{\alpha,\gamma}<0}} x_{\gamma}^{-b^{\mathcal{T}}_{\alpha,\gamma}}.
\end{align}

In this article we are mainly interested in cluster algebras attached to the following two marked oriented surfaces which are also shown in Figure \ref{Figure:AD}:

\begin{ex}
\label{ref:MainExamples}
\begin{itemize}
\item[(a)] (Cluster algebra of type $A$) In the first case we pick a natural number $n\geq 0$ and consider the disk $\Sigma=D$ with a set $M\subseteq \partial D$ of $n+3$ pairwise different points on the boundary. We may realize $(\Sigma,M)$ as an $(n+3)$-gon in $\mathbb{R}^2$, and we label the marked points in this order by the natural numbers $0,1,\ldots,n+2$. Then every arc is isotopic to a line segment $(ij)$ connecting two different points $i,j\in M$, and such an arc has to be tagged plain at both endpoints. A possible triangulation consists of the boundary arcs together with the diagonals $(0j)$ with $j\in\{2,3,\ldots,n+1\}$. The corresponding cluster algebra $\mathcal{A}(\Sigma,M)$ contains only finitely many cluster variables because the marked oriented surface contains only finitely many tagged arcs. In fact, the cluster algebra is of type $A_n$ and for this reason we will denote it by $\mathcal{A}(A_n)$. For an arc $\alpha=(ij)$ we sometimes write $x_{ij}$ or $x_{i,j}$ instead of $x_{\alpha}$. Note that the elements $x_{ij}$ attached to boundary arcs $(i,j)\in\{(r,r+1)\colon 0\leq r\leq n+2\}\cup\{(0,n+3)\}$ are frozen variables whereas the elements $x_{ij}$ attached to flippable inner arcs $(i,j)$ are cluster variables.

\item[(b)] (Cluster algebra of type $D$) The second case of interest is the disk $\Sigma=D$ with one puncture $0$ and $n$ marked points $1,2,\ldots,n$ on $\partial D$. We orient the surface counterclockwise as a subset of $\mathbb{R}^2$ and we assume that $1,2,\ldots,n$ lie on $\partial D$ in this order. Note that for two points $i,j\in\partial D$ there are up to isotopy two arcs between $i$ and $j$, and both arcs have to be tagged plain at both endpoints. If the two arcs are given as the images $\alpha=\operatorname{im}(a)$ and $\beta=\operatorname{im}(b)$ of two curves $a,b\colon [0,1]\to \Sigma$ such that $a(0)=b(1)=i$, $a(1)=b(0)=j$ and the concatenation $a \ast b$ is a positively oriented loop around the marked point $i$, then we abbreviate $\alpha=(ij)$ and $\beta=(ji)$. For any point $j\in \partial M$ there is only one ordinary arc from $0$ to $j$ up to isotopy. The arc can be tagged at the puncture $0$ in two ways. In this case we denote by $(0j)$ the plain tagged arc between $0$ and $j$, and by $(j0)$ the notched tagged arc between $0$ and $j$. A possible tagged triangulation consists of the boundary arcs together with the plain tagged arcs $(0j)$ connecting $0$ with the vertices $j\in\{1,2,\ldots,n\}$. As in the previous example the marked oriented surface admits only finitely many tagged arcs. Hence the cluster algebra admits only finitely many cluster variables. In fact, it is a cluster algebra of type $D_n$ and we will denote by $\mathcal{A}(D_n)$ in the following sections. For a tagged arc $\alpha=(ij)$ we sometimes write $x_{ij}$ or $x_{i,j}$ instead of $x_{\alpha}$. As before, boundary arcs correspond to frozen variables whereas flippable inner arcs correspond to mutable cluster variables.

\end{itemize}
\end{ex}

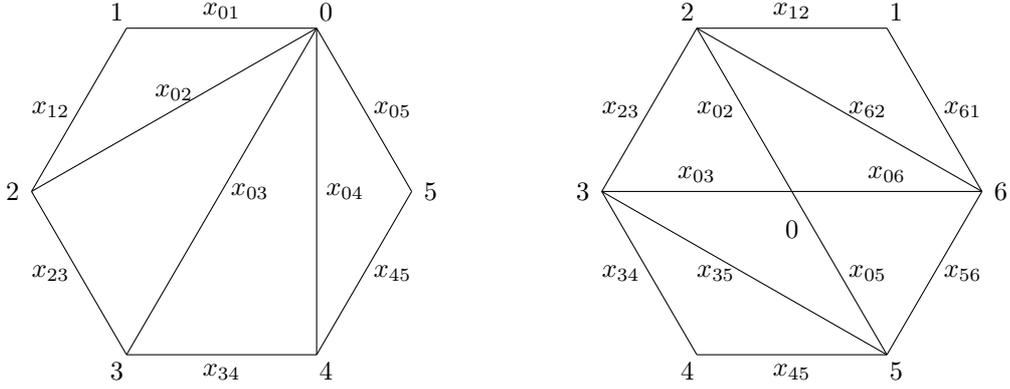
\begin{figure}
\begin{center}
\newcommand{\dist}{3}
\newcommand{\sca}{1.1}
\begin{tikzpicture}[auto,scale=2.5]
    \node at (\sca*0.5,\sca*0.866) {$0$};
    \node at (-\sca*0.5,\sca*0.866) {$1$};
    \node at (-\sca,0) {$2$};
    \node at (-0.5*\sca,-0.866*\sca) {$3$};
    \node at (0.5*\sca,-0.866*\sca) {$4$};
    \node at (1*\sca,0) {$5$};

    \path[-] (1,0) edge node[above,right] {$x_{05}$} (0.5,0.866);
    \path[-] (0.5,0.866) edge node[above] {$x_{01}$} (-0.5,0.866);
    \path[-] (-0.5,0.866) edge node[above,left] {$x_{12}$} (-1,0);
    \path[-] (-1,0) edge node[below,left] {$x_{23}$} (-0.5,-0.866);
    \path[-] (-0.5,-0.866) edge node[below] {$x_{34}$} (0.5,-0.866);
    \path[-] (0.5,-.866) edge node[below,right] {$x_{45}$} (1,0);
    
    \path[-] (-1,0) edge node[above] {$x_{02}$} (0.5,0.866);
    \path[-] (-0.5,-0.866) edge node[right] {$x_{03}$} (0.5,0.866);
    \path[-] (0.5,-0.866) edge node[right] {$x_{04}$} (0.5,0.866);
    
    \node at (\dist+\sca*0.5,\sca*0.866) {$1$};
    \node at (\dist-\sca*0.5,\sca*0.866) {$2$};
    \node at (\dist-\sca,0) {$3$};
    \node at (\dist,-0.2) {$0$};
    \node at (\dist-0.5*\sca,-0.866*\sca) {$4$};
    \node at (\dist+0.5*\sca,-0.866*\sca) {$5$};
    \node at (\dist+1*\sca,0) {$6$};
    
    \path[-] (\dist+1,0) edge node[above,right] {$x_{61}$} (\dist+0.5,0.866);
    \path[-] (\dist+0.5,0.866) edge node[above] {$x_{12}$} (\dist-0.5,0.866);
    \path[-] (\dist-0.5,0.866) edge node[above,left] {$x_{23}$} (\dist-1,0);
    \path[-] (\dist-1,0) edge node[below,left] {$x_{34}$} (\dist-0.5,-0.866);
    \path[-] (\dist-0.5,-0.866) edge node[below] {$x_{45}$} (\dist+0.5,-.866);
    \path[-] (\dist+0.5,-.866) edge node[below,right] {$x_{56}$} (\dist+1,0);
    
    \path[-] (\dist+1,0) edge node[above,right] {$x_{62}$} (\dist-0.5,0.866);
    \path[-] (\dist,0) edge node[above,left] {$x_{02}$} (\dist-0.5,0.866);
    \path[-] (\dist,0) edge node[above] {$x_{03}$} (\dist-1,0);
    \path[-] (\dist-1,0) edge node[below,left] {$x_{35}$} (\dist+0.5,-0.866);
    \path[-] (\dist,0) edge node[below,right] {$x_{05}$} (\dist+0.5,-.866);
    \path[-] (\dist,0) edge node[below,above] {$x_{06}$} (\dist+1,0);

\end{tikzpicture}
\end{center}
\caption{Seeds of two surface cluster algebras associated with a hexagon}
\label{Figure:AD}
\end{figure}

The exchange relation (\ref{Exchange}) has an interesting geometric meaning. Suppose that the ordinary arcs $\alpha$, $\beta$, $\gamma$ and $\delta$ form a quadrilateral with no marked points in the interior. We denote the two diagonal arcs inside the quadrilateral by $\epsilon$ and $\digamma$. Then the exchange relation reads $x_{\alpha}x_{\gamma}+x_{\beta}x_{\delta}=x_{\epsilon}x_{\digamma}$. We may view this relation as an analogue of Ptolemy's theorem from Euclidean geometry: if four points $A,B,C,D\in\mathbb{R}^2$ lie on a circle in this order, then the Euclidean distances obey the relation $\lvert AB\rvert\cdot\lvert CD\rvert+\lvert BC\rvert\cdot\lvert DA\rvert=\lvert AC\rvert\cdot\lvert BD\rvert$. Note that in the cluster algebra less relations hold than in Euclidean geometry. For example, in the above situation we have $\vert AC\vert/\vert BD\vert=(\vert BC\vert\cdot\vert CD\vert+\vert DA\vert\cdot\vert AB\vert)/(\vert AB\vert\cdot\vert BC\vert+\vert CD\vert\cdot\vert DA\vert)$, but the corresponding cluster variables do not satisfy a similar relation.

Fomin--Thurston \cite{FT} give the cluster variables a geometric interpretation in terms of Penner's \textit{lambda lengths} from Teichm\"uller theory. Recall that the \textit{Teichm\"uller space} $T(\Sigma,M)$ is the set of isotopy classes of complete hyperbolic metrics of finite area on $\Sigma\backslash M$ with geodesic boundary on $\partial\Sigma\backslash M$. Here we say that two hyperbolic metrics are \textit{isotopic} if we can obtain one metric from the other by a diffeomorphism of $\Sigma$ that fixes $M$ and is homotopic to the identity. Suppose that $m\in T(\Sigma,M)$ is such a hyperbolic metric. Then every arc is isotopic to a unique geodesic. This geodesic may have infinite hyperbolic length if its endpoints belong to $M$. In contrast, Penner's lambda lengths assign to every arc between marked points $i,j\in M$ a finite length, using the notion of horocycles. Here, a \textit{horocyle} $h$ around a marked point $i$ (with respect to the hyperbolic metric $m$) is a curve which is perpendicular to every geodesic ending in $i$. Note that a horocycle is a closed curve and it is determined by its center $i$ and its hyperbolic length $l(h)$. A \textit{decoration} of a hyperbolic metric $m\in T(\Sigma,M)$ is a collection of horocycles $\mathbf{h}=(h_i)_{i\in M}$ around the marked points $i\in M$. The \textit{decorated Teichm\"uller space} $\widetilde{T}(\Sigma,M)$ is the set of all pairs $(m,\mathbf{h})$ where $m$ is an isotopy class of a hyperbolic structure $m\in T(\Sigma,M)$ and $\mathbf{h}$ a decoration. Suppose that $(m,\mathbf{h})\in \widetilde{T}(\Sigma,M)$ is such a decorated metric. An arc between two marked points $i,j\in M$ is isotopic to a unique geodesic $\gamma$ (with respect to $m$). Let us denote the intersection of $\gamma$ with the horocycle $h_i$ by $i'$ and the intersection of $\gamma$ with the horocycle $h_j$ by $j'$. Moreover, we denote by $l(\gamma)$ the signed hyperbolic length of the curve from $i'$ to $j'$ along $\gamma$, with a positive sign if $h_i$ and $h_j$ do not intersect and a negative sign otherwise. The \textit{lambda length} of $\gamma$ (with respect to the decorated metric) is defined as $\lambda_{\gamma}=\operatorname{exp}(l(\gamma)/2)$.

Lambda lengths satisfy many remarkable properties. A fundamental theorem of Fomin and Thurston asserts that for any choice of a decorated metric the lambda lengths $\lambda_{\gamma}\in\mathbb{R}$ attached to the plain tagged arcs $\gamma$ satisfy the same relations as the corresponding cluster variables $x_{\gamma}\in\mathcal{A}(\Sigma,M)$. For example, suppose that the ordinary arcs $\alpha$, $\beta$, $\gamma$ and $\delta$ form a quadrilateral whose vertices are marked points in $M$ such that there are no marked points in the interior of the quadrilateral. Then the lambda lengths obey the Ptolemy relation $\lambda_{\alpha}\lambda_{\gamma}+\lambda_{\beta}\lambda_{\delta}=\lambda_{\epsilon}\lambda_{\digamma}$ where $\epsilon$ and $\digamma$ are the diagonal arcs inside the quadrilateral. In the light of the above considerations we may identify the cluster variable $x_{\gamma}$ attached to a plain tagged arc with the function $\widetilde{T}(\Sigma,M)\to\mathbb{R}$ induced by the lambda length of $\gamma$. To model tagged arcs, let us fix a decorated metric $(m,\mathbf{h})\in \widetilde{T}(\Sigma,M)$. For every $i\in M$ let us call the horocycle $\overline{h}_i$ around $i$ with $l(h_i)\cdot l(\overline{h}_i)=1$ the \textit{conjugated horocycle}. Suppose that the arc $\gamma$ is tagged notched at one endpoint $i$ (or at both endpoints $i,j$). Then we define the lambda length of $\gamma$ with respect to $(m,\mathbf{h})$ to be the lambda length of $\gamma$ with respect to the decorated metric $(m,\mathbf{h}')$ where the decoration $\mathbf{h}'$ is obtained from $\mathbf{h}$ by replacing $h_i$ with $\overline{h}_i$ (and also  $h_j$ with $\overline{h}_j$). As before, the lambda lengths $\lambda_{\gamma}\in\mathbb{R}$ satisfy the same relations as the corresponding cluster variables $x_{\gamma}\in\mathcal{A}(\Sigma,M)$. Thus we may identify the cluster variable $x_{\gamma}$ attached to a tagged arc $\gamma$ with the function $\widetilde{T}(\Sigma,M)\to\mathbb{R}$ induced by the lambda length of $\gamma$.

The following relations will become useful later: 

\begin{rem}[Properties of lambda lengths]
	\label{rem:LambdaLengths}
\begin{itemize}	
\item[(a)] Let $\eta$ be the loop with endpoint $0$ around a self-folded triangle with unique puncture $1$, compare the left and middle part of Figure \ref{Figure:SelfFoldedTriangle}. Moreover, let $\alpha$ be the plain tagged arc between $0$ and $1$ and $\overline{\alpha}$ the notched tagged arc between $0$ and $1$. In this situation we have $\lambda_{\eta}=l(h_1)\cdot\lambda_{\alpha}^2=l(\overline{h}_1)\cdot\lambda_{\overline{\alpha}}^2$,
see Fomin--Thurston \cite[Proposition 7.9 and Proof of Proposition 7.10]{FT}. This relation has interesting consequences. On the one hand we can conclude that $1=l(h_1)\cdot l(\overline{h}_1)=\lambda_{\eta}^2/(\lambda_{\alpha}^2\cdot\lambda_{\overline{a}}^2)$ so that $\lambda_{\eta}=\lambda_{\alpha}\cdot\lambda_{\overline{\alpha}}$, see \cite[Proposition 7.10 and Lemma 8.2]{FT}. On the other hand, we can conclude that $\lambda_{\overline{\alpha}}^2/\lambda_{\alpha}^2=l(h_1)/l(\overline{h}_1)=l(h_1)^2$ so that the quotient $\lambda_{\overline{\alpha}}/\lambda_{\alpha}$ is equal to the hyperbolic length of the horocycle $h_1$.
\item[(b)] As in the right side of Figure \ref{Figure:SelfFoldedTriangle} we consider a digon with vertices $1,2$ and edges $\alpha,\beta$. We assume that the digon contains exactly one puncture $0$. We denote by $\gamma$ the plain tagged arc from $0$ to $1$, by $\delta$ the plain tagged arc from $0$ to $2$, and by $\eta$ the loop around $0$ with endpoint $1$. A result of Fomin--Thurston \cite[Corollary 7.7]{FT} asserts that $\lambda_{\delta}\cdot\lambda_{\eta}=\lambda_{\gamma}\cdot(\lambda_{\alpha}+\lambda_{\beta})$. Note that a similar result holds true for the loop around $0$ with endpoint $2$.
\end{itemize}
\end{rem}

\section{Determinantal identities}
\label{Sec:Det}

\subsection{Determinantal identities for Euclidean distance matrices}

Let $k,m\geq 0$ be natural numbers and suppose that $P_0,P_1,\ldots,P_k$ are situated arbitrarily in $\mathbb{R}^m$ (or in a general metric space). For two indices $i,j\in\{0,1,\ldots,k\}$ we denote by $d_{ij}=\lvert P_iP_j\rvert\in\mathbb{R}^{\geq 0}$ the Euclidean distance between $P_i$ and $P_j$. The objective of \textit{distance geometry} is to reconstruct as much information about the configuration of the points $P_0,P_1,\ldots,P_k$ as possible when all the mutual distances $d_{ij}$ are given. Distance geometry has many applications, for example in GPS navigation. As it turns out, determinants epitomize many properties.

\begin{ex}[Cayley--Menger determinants]
\label{Ex:CM}
\begin{itemize}
\item[(a)] Let $P_0,P_1,P_2$ be points in $\mathbb{R}^2$. We abbreviate $a=\lvert P_0P_1\rvert$, $b=\lvert P_1P_2\rvert$ and $c=\lvert P_2P_0\rvert$. By Heron's formula the determinant
\begin{align*}
D=\left\lvert\begin{matrix}
0&a^2&c^2&1\\
a^2&0&b^2&1\\
c^2&b^2&0&1\\
1&1&1&0\\
\end{matrix}\right\rvert
=-(a+b+c)(a+b-c)(a-b+c)(-a+b+c)
\end{align*}
is equal to $-16A^2$ where $A$ is the area of the triangle $\Delta P_0P_1P_2$. Especially $P_0$, $P_1$ and $P_2$ are collinear if and only if $D=0$. More generally, let $P_0,P_1,\ldots,P_{k}$ be points in $\mathbb{R}^k$. Cayley and Menger consider the $(k+2)\times(k+2)$ matrix
\begin{align*}
CM=\left(\begin{matrix}
0&d_{01}^2&d_{02}^2&d_{03}^2&\cdots&d_{0k}^2&1\\
d_{01}^2&0&d_{12}^2&d_{13}^2&\cdots&d_{1k}^2&1\\
d_{02}^2&d_{12}^2&0&d_{23}^2&\cdots&d_{2k}^2&1\\
d_{03}^2&d_{13}^2&d_{23}^2&0&\cdots&d_{3k}^2&1\\
\vdots&\vdots&\vdots&\vdots&\ddots&\vdots&\vdots\\
d_{0k}^2&d_{1k}^2&d_{2k}^2&d_{3k}^2&\cdots&0&1\\
1&1&1&1&\cdots&1&0
\end{matrix}\right).
\end{align*}
Cayley's theorem asserts that $P_0,P_1,\ldots,P_{k}$ lie in a common hyperplane if and only if $\operatorname{det}(CM)=0$. 
\item[(b)] Suppose that $k\geq 3$ and that the points $P_0,P_1,\ldots,P_{k}$ lie in $\mathbb{R}^{k-1}$. Let us remove the last row and the last column of the previous matrix. In other words we consider the $(k+1)\times(k+1)$ principal minor $CM_{I,I}$ of the Cayley--Menger matrix supported on the index set $I=\{0,1,2,\ldots,k\}$. A folklore theorem asserts that $P_0,P_1,\ldots,P_{k}$ lie on a common sphere or in a common hyperplane if and only if $CM_{I,I}=0$, see for example Pak \cite[Exercise 34.5]{P}. 
\end{itemize}
\end{ex}

\subsection{Determinantal identities for surface cluster algebras of type A}

Let $n\geq 1$ be a natural number. In this section we consider the cluster algebra $\mathcal{A}(A_n)$ from Example \ref{ref:MainExamples}. As it turns out, several determinants are incarnations of cluster variables in $\mathcal{A}(A_n)$.

\begin{rem}[Descriptions of cluster variables as determinants]
\begin{itemize}
\item[(a)] As an algebra $\mathcal{A}(A_n)$ is isomorphic to the homogeneous coordinate ring of the Grassmannian variety $\operatorname{Gr}_2(\mathbb{Q}^{n+3})$ of $2$-dimensional subspaces of a $(n+3)$-dimensional vector space. An element in $U\in\operatorname{Gr}_2(\mathbb{Q}^{n+3})$ is uniquely determined by its Pl\"ucker coordinates i.\,e. the $2\times 2$ minors of a $2\times (n+3)$ matrix whose rows span the subspace $U$. Such a Pl\"ucker coordinate $P_{ij}$ is given by a choice of two columns $i,j\in\{0,1,\ldots,n+2\}$. The exchange relations among the frozen and cluster variables $x_{ij}$ are the same as the Pl\"ucker relations among the coordinates $P_{ij}$. 

\item[(b)] Let $\mathcal{T}$ be the star-shaped triangulation $\mathcal{T}=\mathcal{T}^b\cup\{(i,n+2)\colon 1\leq i\leq n\}$ of $(\Sigma,M)$. There is an explicit description of the Laurent expansions of the cluster variables with respect to the initial seed $(\mathbf{x},B^{\mathcal{T}})$ attached to $\mathcal{T}$. To simplify notation let us denote by $x_{i}'=x_{i-1,i+1}$ the cluster variable obtained by mutating $(\mathbf{x},B^{\mathcal{T}})$ at the index $(i,n+2)\in\mathcal{T}$, or equivalently by flipping $\mathcal{T}$ at the arc $(i,n+2)$. Then we may write the non-initial cluster variables as determinants of tridiagonal matrices as follows. We claim that for all indices $i,j$ with $1\leq i\leq j\leq n$ we have
\begin{align}
\label{Eqn:ClusterVarDet}
x_{i-1,j+1}=M\cdot\left\lvert\begin{matrix}
x_i'&x_{i+1,i+2}&0&0&\cdots&0&0\\
x_{i-1,i}&x_{i+1}'&x_{i+2,i+3}&0&\cdots&0&0\\
0&x_{i,i+1}&x_{i+2}'&x_{i+3,i+4}&\cdots&0&0\\
0&0&x_{i+1,i+2}&x_{i+3}'&\cdots&0&0\\
\vdots&\vdots&\vdots&\vdots&\ddots&\vdots&\vdots\\
0&0&0&0&\cdots&x_{j-1}'&x_{j,j+1}\\
0&0&0&0&\cdots&x_{j-2,j-1}&x_j'\\
\end{matrix}\right\rvert
\end{align}
where indices are read modulo $n+3$ and $M=x_{i,i+1}\cdot x_{i+1,i+2}\cdot\ldots\cdot x_{j-1,j}$ is a monomial. Note that the entries on the main diagonal are cluster variables, the entries on first diagonal below and above the main diagonal are frozen variables and all other entries are zero. We can prove this claim by mathematical induction on $j$ for every fixed $i$. The statement is true for $j=i$ by construction and easy to verify for $j=i+1$ using the Ptolemy relation $x_i'\cdot x_{i+1}'-x_{i-1,i}\cdot x_{i+1,i+2}=x_{i-1,i+2}\cdot x_{i,i+1}$ for the quadrilateral with vertices $i-1,i,i+1,i+2$. The induction step follows from the Ptolemy relation $x_{i-1,j+1}\cdot x_{j-1,j}=x_{i-1,j}\cdot x_{j-1,j+1}-x_{i-1,j-1}\cdot x_{j,j+1}$ for the quadrilateral with vertices $i-1,j-1,j,j+1$ in combination with the Laplace expansion of the determinant. 

Equation (\ref{Eqn:ClusterVarDet}) illustrates Fomin--ZelevinskyÕs Laurent phenomenon: Every cluster variable $x_i'$ belongs to the ring $\mathbb{Z}[\mathbf{x}^{\pm 1}]$ by definition, hence an expansion of the determinant allows us to write every $x_{i-1,j+1}$ as an element in the Laurent polynomial ring $\mathbb{Z}[\mathbf{x}^{\pm 1}]$ of the initial cluster. Moreover, it visualizes the inclusion $x_{i-1,j+1}\in\mathbb{Z}[x_m,x_{m}'\colon 1\leq m\leq n]$, compare Berenstein--Fomin--Zelevinsky's equality of cluster algebra and lower bound \cite[Theorem 1.18]{BFZ}.
\end{itemize}
\end{rem}

Baur--Marsh's \cite{BM} main result concerns the symmetric $(n+3)\times(n+3)$ matrix $BM$ with entries in $\mathcal{A}(A_n)$ for which $BM_{i,j}=x_{i,j}$ holds for all marked points $i,j\in M$. Here we interpret $x_{i,i}$ as zero (because the loop around  $i$ is contractible). Note that Baur and Marsh write the equation in a slightly different form. The different sign in their formula is due to the usage of minors instead of cluster variables.

\begin{theorem}[Baur--Marsh]
\label{thm:BaurMarsh}
The equality
\begin{align*}
\operatorname{det}(BM)=-(-2)^{n+1}\cdot x_{0,1}\cdot x_{1,2}\cdot x_{2,3}\cdot\ldots\cdot x_{n+1,n+2}\cdot x_{n+2,0}
\end{align*}
holds true.
\end{theorem}

Motivated by the two determinants in Euclidean geometry from Remark \ref{Ex:CM} we consider two variations of Baur--Marsh's determinant. We call the symmetric $(n+4)\times (n+4)$ matrix $CM$ with entries  
\begin{align*}
CM_{i,j}=\begin{cases}
0,&\textrm{if } i=j;\\
x_{i,j}^2,&\textrm{if }i\neq j\textrm{ and }0\leq i,j\leq n+2;\\
1,&\textrm{if } i\neq j\textrm{ and }($i=n+3$\textrm{ or }$j=n+3$);
\end{cases}
\end{align*}
in the ring $\mathcal{A}(A_n)$ the \textit{Cayley--Menger cluster matrix of type $A$}. We are also interested in minors of $CM$. Note that if we pick the index set $I=\{0,1,2,\ldots,n+2\}$, then the principal submatrix $PM=CM\vert_{I\times I}$ is the equal to the Hadamard product of the Baur--Marsh matrix with itself. That is, both matrices are indexed by the same set $I$ and for all indices $i,j\in I$ we have $PM_{i,j}=BM_{i,j}^2$. Some authors use the term \textit{entrywise product} instead of \textit{Hadamard product}.

\begin{theorem}[Cayley--Menger cluster matrix of type A]
\label{Thm:CM}
\begin{itemize}
\item[(a)] If $I,J\subseteq \{0,1,\ldots,n+2\}$ are subsets of cardinality $\vert I\vert=\vert J\vert=k\geq 4$, then the minor of the Caley--Menger cluster matrix of type A supported on the index set $I,J$ obeys the equation $CM_{I,J}=PM_{I,J}=0$.
\item[(b)] If $n\geq 2$, then the equation $\operatorname{det}(CM)=0$ holds true. Especially, the matrix $CM$ is singular i.\,e. it does not admit an inverse matrix.
\end{itemize}
\end{theorem}

\begin{proof}
\begin{itemize}
\item[(a)] It is enough to prove the statement for $k=4$. Assume that $I,J\subseteq \{0,1,\ldots,n+2\}$ are two subsets of cardinality $k=4$. For every $n'>n$ there is an inclusion $\iota_{n,n'}\colon \mathcal{A}(A_n)\hookrightarrow \mathcal{A}(A_{n'})$ of rings, and we may view $CM$ as a matrix with entries in $\mathcal{A}(A_{n'})$. The map $\iota_{n,n'}$ is in fact an injective rooted cluster morphism, compare Assem--Dupont--Schiffler \cite[Definition 2.2 and Example in the beginning of Section 4.1]{ADS}. Hence we may assume that there is an index $r\in \{0,1,\ldots,n+2\}\backslash (I\cup J)$ without loss of generality.

Let $\mathcal{F}$ be the ambient field of the cluster algebra $\mathcal{A}(A_{n})$. For all indices $i,j\in \{0,1,\ldots,n+2\}\backslash\{r\}$ we put $\widetilde{x}_{ij}=x_{ij}/(x_{ri}\cdot x_{rj})\in\mathcal{F}$. The Ptolemy relation implies that $\widetilde{x}_{ij}+\widetilde{x}_{jk}=\widetilde{x}_{ik}$ whenever $i,j,k\in\{0,1,\ldots,n+2\}\backslash\{r\} $ are pairwise different indices such that $r,i,j,k$ lie on the boundary of $\Sigma$ in this order. The relation is also true if $i=j$ or $j=k$. Let us put $s=r+1$. The previous relation yields $\widetilde{x}_{ij}=\pm(\widetilde{x}_{is}-\widetilde{x}_{sj})$ for all $i\in I$ and $j\in J$. Especially we have $\widetilde{x}_{ij}^2=\widetilde{x}_{is}^2-2\widetilde{x}_{is}\widetilde{x}_{sj}+\widetilde{x}_{sj}^2$ for all $i\in I$ and $j\in J$.

Let $\widetilde{PM}_{I,J}$ be the $4\times 4$ matrix indexed by $I\times J$ with entries $\widetilde{PM}_{ij}=\widetilde{x}_{ij}^2$ for all $i\in I$ and $j\in J$. Let us introduce the column vectors $v_I=(\widetilde{x}_{is})_{i\in I}$ and $w_I=(\widetilde{x}_{is}^2)_{i\in I}$ and the row vectors $v_J=(\widetilde{x}_{sj})_{j\in J}$ and $w_J=(\widetilde{x}_{sj}^2)_{j\in J}$. Then $\widetilde{PM}_{I,J}$ is a $\mathcal{F}$-linear combination of the $4\times 4$ matrices $w_I \cdot (1,1,1,1)$, $v_I\cdot v_J$, and $(1,1,1,1)^T\cdot w_J$ of rank $1$. We conclude that
\begin{align*}
\operatorname{det}(PM_{I,J})=\left(\prod_{i\in I}x_{ri}\right)^2\cdot\left(\prod_{j\in J}x_{rj}\right)^2\cdot\operatorname{det}(\widetilde{PM}_{I,J})=0. 
\end{align*}

\item[(b)] The second statement follows from the first by the following observation. If $A$ is any $(n+3)\times(n+3)$ matrix and $u$ is a (column) vector of size $n+3$ and $\lambda$ is a scalar, then 
\begin{align*}
\left\lvert\begin{matrix}
A & u\\
u^T & \lambda
\end{matrix}\right\rvert = \lambda\cdot \operatorname{det}(A) -u^T\cdot\operatorname{adj}(A)\cdot u
\end{align*}
where $\operatorname{adj}(A)$ is the adjoint matrix of $A$. In our case we have $\lambda=0$ and $\operatorname{adj}(A)=\operatorname{adj}(PM)=0$, since up to sign the entries in $\operatorname{adj}(PM)$ are $(n+2)\times(n+2)$ minors of $PM$. All of them vanish by part (a) because $n\geq 2$.
\end{itemize}
\end{proof}

\subsection{Determinantal identities for surface cluster algebras of type D}

Let $n\geq 4$ be a natural number and consider the cluster algebra $\mathcal{A}(D_n)$ from Example \ref{ref:MainExamples} attached to a polygon with vertex set $\{1,2,\ldots,n\}$ with one puncture $0$ in its interior. The \textit{generalized Baur--Marsh matrix of type $D$} is the $(n+1)\times(n+1)$ matrix $\BMD$ with entries in $\mathcal{A}(D_n)$ for which $\BMD_{i,j}=x_{i,j}$ holds for all marked points $i,j\in\{0,1,\ldots,n\}$. As before we interpret $x_{ii}$ as zero (because the loop around  $i$ is contractible). Note that $\BMD$ is not symmetric because the cluster variables $x_{i,j}\neq x_{j,i}$ are attached to two different arcs between $i$ and $j$ whenever $i,j\geq 1$ are different indices.

Remark \ref{rem:LambdaLengths}, which summarizes some properties about lambda lengths Fomin and Thurston established in their work, leads to the following conclusions: 

\begin{prop}[Properties of cluster variables]
\label{rem:Useful}
Let $i,j\in\{1,2,\ldots,n\}$ be two different indices corresponding to marked points on the boundary of $\Sigma$. The entries of the generalized Baur--Marsh matrix of type $D$ obey the relations
\begin{align*}
x_{ij}+x_{ji}=x_{i0}\cdot x_{0j}=x_{0i}\cdot x_{j0} 
\end{align*}
inside the cluster algebra $\mathcal{A}(D_n)$.
\end{prop}

\begin{proof}
Inside the ambient field $\mathbb{Q}(x_1,x_2,\ldots,x_n)$ of $\mathcal{A}(D_n)$ the identity $x_{i0}/x_{0i}=x_{j0}/x_{0j}$ holds, because by Remark \ref{rem:LambdaLengths} (a) the quotients of the corresponding lambda lengths are both equal to the hyperbolic length of the conjugated horocycle around the puncture $0$. Especially, the relation $x_{i0}x_{0j}=x_{0i}x_{j0}$ holds inside the cluster algebra $\mathcal{A}(D_n)$.

Note that the product $x_{i0}\cdot x_{0i}$ corresponds to the lambda length of the loop around the puncture $0$ with endpoint $i$ by virtue of Remark \ref{rem:LambdaLengths} (a). From part (b) of the same remark we conclude that the relation $x_{0i}\cdot(x_{ij}+x_{ji})=x_{0j}\cdot x_{i0}\cdot x_{0i}$ holds. It follows that $x_{ij}+x_{ji}=x_{i0}\cdot x_{0j}$ because the cluster algebra $\mathcal{A}(D_n)$ does not contain zero divisors.
\end{proof}

\begin{theorem}[Generalized Baur--Marsh matrix of type D] 
\label{Thm:TypeD}
The equality
\begin{align*}
\operatorname{det}(\BMD)= (-1)^{n}\cdot \prod_{i=1}^{n} x_{0i}\cdot x_{i0}
\end{align*}
holds true inside the cluster algebra $\mathcal{A}(D_n)$.
\end{theorem}

\begin{proof}
Let us denote the rows of the matrix by $r_0,r_1,\ldots,r_n$. For every index $i\in\{2,3,\ldots,n\}$ we perform a row operation which replaces row $r_i$ with the linear combination $\widetilde{r}_i=x_{01}\cdot r_i+x_{1i}\cdot r_0-x_{0i}\cdot r_1$. This row operation changes the determinant by the factor $x_{01}$. We can compute the entries of $\widetilde{r}_i$ explicitly:
\begin{itemize}

\item[(i)] The entry $\left(\widetilde{r}_i\right)_0=x_{01}\cdot x_{i0}+0-x_{0i}\cdot x_{10}=0$ vanishes due to Proposition \ref{rem:Useful}.

\item[(ii)] The same proposition implies $\left(\widetilde{r}_i\right)_1=x_{01}\cdot x_{i1}+x_{1i}\cdot x_{01}-0=x_{01}\cdot(x_{i1}+x_{1i})=x_{01}^2\cdot x_{i0}$.

\item[(iii)] We claim that $\left(\widetilde{r}_i\right)_j=x_{01}\cdot x_{0i} \cdot x_{j0}$ whenever $j\in\{2,3,\ldots,i-1\}$. To verify the claim, note that we have $\left(\widetilde{r}_i\right)_j=x_{01}\cdot x_{ij}+x_{1i}\cdot x_{0j}-x_{0i}\cdot x_{1j}$ by construction. The Ptolemy relation for the quadrilateral with vertices $0,1,j,i$ implies $x_{1i}\cdot x_{0j}-x_{0i}x_{1j}=x_{01}\cdot x_{ji}$ so that we obtain $\left(\widetilde{r}_i\right)_j=x_{01}\cdot x_{ij}+x_{01}\cdot x_{ji}$. Now the claim follows from Proposition \ref{rem:Useful}. 

\item[(iv)] The entry $\left(\widetilde{r}_i\right)_i=0+x_{1i}\cdot x_{0i}-x_{01}\cdot x_{1i}=0$ vanishes by construction.

\item[(v)] The entry $\left(\widetilde{r}_i\right)_j=x_{01}\cdot x_{ij}+x_{0j}\cdot x_{1i}-x_{0i}\cdot x_{1j}=0$ vanishes whenever $j\in\{i+1,i+2,\ldots,n\}$ due to the Ptolemy relation for the quadrilateral with vertices $0,1,i,j$.

\end{itemize}

We denote by $\BMDmod$ the resulting matrix obtained after performing all these row operations. In other words, $\BMDmod$ is the matrix with rows $r_0,r_1,\widetilde{r}_2,\ldots,\widetilde{r}_n$. The matrix $\BMDmod$ is sparse. The only non-zero entry in row $2$ is $\widetilde{r}_{21}$, the only non-zero entries row $3$ are $\widetilde{r}_{32}$ and $\widetilde{r}_{31}$, and the only non-zero entries in row $n$ are $\widetilde{r}_{n,n-1},\widetilde{r}_{n,n-2},\ldots,\widetilde{r}_{n,1}$. Together with the above considerations we can conclude that
\begin{align*}
(-1)^{n}\cdot\operatorname{det}(\BMDmod)=-\BMDmod_{01,0n}\cdot \prod_{i=2}^{n}\widetilde{r}_{i,i-1} = \left(x_{1,0}\cdot x_{0,n}\right)\cdot \prod_{i=2}^{n}\left(x_{0,1}\cdot x_{0,i}\cdot x_{i-1,0}\right)=\left(x_{0,1}\right)^{n-1}\cdot \prod_{i=1}^{n}\left(x_{0,i}\cdot x_{i,0}\right),
\end{align*}
where $\BMDmod_{01,0n}=-x_{1,0}\cdot x_{0,n}=-x_{0,1}\cdot x_{n,0}$ is the minor supported on rows $0,1$ and columns $0,n$. As noted above we have $\operatorname{det}(\BMDmod)=(x_{0,1})^{n-1}\cdot\operatorname{det}(\BMD)$. Hence the theorem follows by cancelling the factor $(x_{0,1})^{n-1}\in \mathcal{A}(D_n)$, which is possible because the cluster algebra $\mathcal{A}(D_n)$ does not contain zero divisors.
\end{proof}

Note that we may interpret the factor $x_{i0}\cdot x_{0i}$ on the right hand side of the theorem as the lambda length of the loop around the puncture $0$ with endpoint $i$.

It would be interesting to find a framework that combines Baur--Marsh's and our setup and admits generalizations to other surface cluster algebras.

\subsection{Acknowledgments}

The author would like to thank Karin Baur and Robert Marsh for valuable discussion about their work about determinants in type $A$ and possible generalizations. The author would like to thank Anna Felikson for useful comments about lambda lengths whose usage shortened proofs of statements in an earlier version of the paper. The author would like to thank Ralf Schiffler for valuable discussions, in particular about variations of the construction of the determinant in type $D$.

\end{document}